\documentclass[oneside,english]{amsart}
\usepackage[T1]{fontenc}
\usepackage[latin9]{inputenc}
\usepackage{amsthm}
\usepackage[pdftex]{graphicx}

\makeatletter
\numberwithin{equation}{section}
\numberwithin{figure}{section}

\newtheorem{thm}{Theorem}[section]
\newtheorem{defi}{Definition}[section]

\newtheorem{lemma}{Lemma}[section]
\newtheorem{prop}{Proposition}[section]
\newtheorem{example}{Example}[section]

\newtheorem{remark}{Remark}[section]

\makeatother

\usepackage{babel}
\begin{document}

\title[\tiny{Hausdorff and box dimension of self-affine set in non-Archimedean field}]{Hausdorff and box dimension of self-affine set in locally compact non-Archimedean field}
\author{Yang Deng}
\address{Department of Mathematics, South China University of Technology, Wushan Road 381, Tiahe District, Guangzhou, 510641, China}
\email{3491374039@qq.com}
\author{Bing Li}
\address{Department of Mathematics, South China University of Technology, Wushan Road 381, Tiahe District, Guangzhou, 510641, China}
\email{scbingli@scut.edu.cn}
\author{Hua Qiu$^*$}
\address{Department of Mathematics, Nanjing University, Nanjing, 210093, P. R. China.}
\email{huaqiu@nju.edu.cn}


\thanks{$^*$ Corresponding author. }
\thanks{The research of Qiu was supported by the National Natural Science Foundation of China, grant 12071213, and the Natural Science Foundation of Jiangsu Province in China, grant BK20211142.}

\date{}

\dedicatory{}
\begin{abstract}
In this paper we consider affine iterated function systems in locally compact non-Archimedean field $\mathbb{F}$. We establish the theory of singular value composition in $\mathbb{F}$ and compute box and Hausdorff dimension of self-affine set in $\mathbb{F}^n$, in generic sense, which is an analogy of Falconer's result for real case. The result has the advantage that no additional assumptions needed to be imposed on the norms of linear parts of affine transformation while such norms are strictly less than $\frac{1}{2}$ for real case, which benefits from the non-Archimedean metric on $\mathbb{F}$. 

\end{abstract}
\maketitle

 \section{Introduction}
For a self-affine set in $\mathbb{R}^n$, Falconer \cite{4} introduced singular value function of linear transformations in its associated iterated function system to determine its Hausdorff dimension. Precisely, given $M\geq 2$, let $T_1,T_2,\cdots,T_M$ be a set of contractive, non-singular linear transformations on
$\mathbb{R}^n$, and  $\textbf{b}=(b_1,b_2,\cdots,b_M)$ be a vector of $M$ points in $\mathbb{R}^n$. We denote by $K(\textbf{b})$ the self-affine  set associated with the affine function system $\{S_1,S_2,\cdots,S_M\}$ of the form $S_i(x)=T_i(x)+b_i$, i.e., $K(\textbf{b})$ is the unique non-empty compact set satisfying 
$$K(\textbf{b})=\bigcup_{i=1}^M S_i(K(\textbf{b})).$$
Providing that $\|T_i\|<\frac{1}{2}$ for each $i$, where $\|T_i\|$ is the operator norm of  $T_i$. Then, there is a number $d(T_1,T_2\cdots,T_M)$ depending on the linear
transformations $T_1,T_2,\cdots,T_M$, such that the invariant set
$K(\textbf{b})$  has Hausdorff dimension
$\min\{n,d(T_1,T_2,\cdots,T_M)\}$ for almost all
$\textbf{b}\in \mathbb{R}^{nM}$ in the sense of
$nM$-dimensional Lebesgue measure.

The critical number $d(T_1,T_2,\cdots,T_M)$ is given in terms of singular values of linear transformations, which is defined as follows. If $T$ is a non-singular linear transformation on $\mathbb{R}^n$, the singular values of $T$ are positive square roots of eigenvalues of $T^tT$, written as $\alpha_1\geq\alpha_2\geq\cdots\geq\alpha_n$, where $T^t$ is the transpose of $T$. They can be viewed as the length of the (mutually perpendicular) principle semiaxes of $T(B)$, where $B$ is the unit ball in $\mathbb{R}^n$. The singular value
function $\phi^s$ is defined by
$$\phi^s(T)=\alpha_1\alpha_2\cdots\alpha_{m-1}\alpha_{m}^{s-m+1},$$
for $0\leq s\leq n$, where $m$ is the smallest integer greater than
or equal to $s$, and
$$\phi^s(T)=(\alpha_1\alpha_2\cdots\alpha_n)^{\frac{s}{n}},$$ for $s>n$. It is easy to prove
that for each $s\geq 0$,  $\phi^s$ is submultiplicative, i.e.
$$\phi^s(TU)\leq\phi^s(T)\phi^s(U),$$ for any 
$T,U\in\mathcal{L}(\mathbb{R}^n,\mathbb{R}^n)$, where $\mathcal{L}(\mathbb{R}^n,\mathbb{R}^n)$ represents all linear transformations from $\mathbb{R}^n$ to $\mathbb{R}^n$.
The critical exponent
$d(T_1,T_2,\cdots,T_M)$ is defined to be the unique non-negative
solution $s$ to equation
\begin{equation}\label{1}
	P(\phi^s):=\lim_{k\rightarrow+\infty}\frac{1}{k}\log\left(\sum_{w_1,\cdots,w_k}\phi^s(T_{w_1}T_{w_2}\cdots T_{w_k})\right)=0, \end{equation}
where the sum is taken over all
finite words $(w_1w_2\cdots w_m)$ of length $m$ with $1\leq
w_i\leq M$, and the number $P(\phi^s)$ in (1.1) is called the topological pressure of $\phi^s$. The submultiplicativity of $\phi^s$ ensures convergence of limit in $(\ref{1})$. Moreover, as a function of $s$, the $P(\phi^s)$ is continuous and
strictly decreasing, and it is greater than $0$
when $s=0$ and is less than $0$ for some large $s$. Hence there exists a
unique $s$ for which the limit equals to $0$.

Falconer \cite{4} proved the dimension result indicated above under the assumption that $\|T_i\|<\frac{1}{3}$ for each $i$,  using potential-theoretic methods.  
Later, Solomyak \cite{10} pointed out that $\|T_i\|<\frac{1}{3}$ could be weakened
to $\|T_i\|<\frac{1}{2}$. The constant $\frac{1}{2}$ is proved to be sharp by Edgar in \cite{2}. Recently, the result was extended to self-affine sets in the sub-Riemannian metric setting of the Heisenberg group by Balogh and Tyson \cite{1}.

In this paper, we consider the Haussdorff dimension of self-affine set in non-Archimedean field, and let $\mathbb{F}$ be a non-Archimedean field, see the next section for its defintion and related properties. Consider the same type affine iterated function system $\{T_1x+b_1,T_2x+b_2,\cdots,T_Mx+b_M\}$
where $T_i$ are non-singular contractive linear transformations on $\mathbb{F}^n$ and
$b_i$ are points in $\mathbb{F}^n$, let $\mu$ be the Haar measure on $\mathbb{F}$, and $\mu^{nM}$ is the $nM$-dimensional product measure generated by $\mu$. Similar to the case on $\mathbb{R}^n$, we use $\textbf{b}$ to denote the vector $(b_1,b_2\cdots,b_M)\in \mathbb{F}^{nM}$ and $K(\textbf{b})$ to denote the associated self-affine set. We will define the similar  critical exponent $d(T_1,T_2,\cdots,T_M)$ (still denoted by $d(T_1,T_2,\cdots,T_M)$ without causing any confusion)  by introducing the analogous concepts of singular value composition and singular value function in non-Archimedean field, and then prove the following analog of the Falconer's
formula under additional condition that $\mathbb{F}$ is a locally compact field. Similarly, we use $\dim_H$ to denote the Hausdorff dimension and use $\dim_B$ to denote the box dimension.

\begin{thm}\label{thm1}
	Let $\mathbb{F}$ be a locally compact non-Archimedean field, $T_i$ be non-singular contractive linear transformations on $\mathbb{F}^n$ for any $1\leq i\leq M$, and $\textbf{b}=(b_1,b_2\cdots,b_M)\in \mathbb{F}^{nM}$. Then, we have\\
	\emph{(i)} $\dim_HK(\textbf{b})\leq \overline{\dim}_BK(\textbf{b})\leq d(T_1, T_2,\cdots,T_M)$ for all $\textbf{b}\in \mathbb{F}^{nM}$; \\
	\emph{(ii)} $\dim_HK(\textbf{b})=\dim_B K(\textbf{b})=\min\{n,d(T_1, T_2,\cdots,T_M)\}$ for $\mu^{nM}$-a.e. $\textbf{b}\in \mathbb{F}^{nM}.$
\end{thm}

In the non-Archimedean field $\mathbb{F}$, we still denote the operator norm of a linear transformations $A$ with $\|A\|$. Then, the results no longer need $\|T_i\|<\frac{1}{2}$ for any $i\leq M$. Actually, from the discreteness of $\|\cdot\|$ in $\mathbb{F}$, $\|T_i\|<1$ means $\|T_i\|\leq\frac{1}{2}$. In particular, when $q=2$, Theorem $\ref{thm1}$ still holds in the case that some $\|T_i\|$ may reach the critical value $\frac{1}{2}$, which is forbidden in the real numbers case.

\section{locally compact non-Archimedean field}
Let $\mathbb{F}$ be an infinite field in which a non-Archimedean valuation is defined for all $ x\in\mathbb{F}$, i.e. a map from $\mathbb{F}$ to the nonnegative real numbers such that, for any $x$, $y\in\mathbb{F}$\\
(i) $\|{x}\|=0$ if and only if $x=0$,\\
(ii) $\|{xy}\|=\|{x}\|\cdot$ $\|{y}\|$, \\
(iii) $\|{x+y}\|\le\max\{\|{x}\|,\|{y}\|\}$.\\
Denote 	$O=\{{a}\in\mathbb{F}:\|{a}\|\le1\},$ and 
$P=\{{a}\in\mathbb{F}:\|{a}\|<1\}.$
Then $O$ forms a ring under addition and multiplication as defined in the field $\mathbb{F}$. Moreover the subset $P$ forms maximal ideal in $O$ since outside $P$ there are only units. Clearly $O/P$ is a field, we call it the residue class field of the filed $\mathbb{F}.$

Moreover, if $\mathbb{F}$ is a locally compact non-Archimedean field, i.e. for any $x\in \mathbb{F}$, there exists a compact set $C$ containing a neighbourhood of $x$. We have following Lemma.
\begin{lemma}
	Let $(\mathbb{F}, \|\cdot\|)$ be a locally compact non-Archimedean field, and $d$ be the nature metric induced by $\|\cdot\|$, then\\
	$\emph{(i)} (\mathbb{F},d)$ is complete, and $O$ is compact.\\
	$\emph{(ii)} O/P$ is a finite set.\\
	$\emph{(iii)} P$ is a principal ideal, i.e. there exists $\pi\in P$ s.t. $P=\pi O$.
\end{lemma}
Denote ${\rm card}(O/P)=q$. 
Usually, we let $\|\pi\|=\frac{1}{q}$, it is a normalization of $\|\cdot\|$. At this point, we have following propostion.
\begin{prop}
	Let $\mu$ be the Haar measure on $\mathbb{F}$, here $\mathbb{F}$ is regarded as a compact group under "$+$", such that $\mu(O)=1$. Then, for any $A\in \mathcal{B}(\mathbb{F})$ and any $ c\in\mathbb{F}$, $\mu(cA)=\parallel{c}\parallel\cdot\mu(A)$, where $\mathcal{B}(\mathbb{F})$ is the Borel algebra in $\mathbb{F}$.
\end{prop}
Propostion 2.1 and Lemma 2.1 are the basic results of algebra, which are proved in \cite{14}. 
Next, we will briefly give two classical examples of locally compact non-Archimedean field, see \cite{7,13} for more details.
\begin{example}
	Let $\mathbb{K}$ be a finite filed with q elements, and $\mathbb{K}((X^{-1}))$ be the field of formal Laurent series, i.e. 
	$$ \mathbb{K}((X^{-1}))=\left\{\sum_{n=n_0}^{+\infty}x_nX^{-n}:x_n\in\mathbb{K}\text{ and }n_0\in\mathbb{Z}\right\} .$$
	Denote $\deg(x)=-\inf\{n\in\mathbb{Z}:x_n\neq 0\}$, where $x=\sum_{n=n_0}^{+\infty}x_nX^{-n}\in\mathbb{K}((X^{-1}))$. In particular, $\deg(0)=-\infty$.
	Define the norm of $x$ to be $\|x\|=q^{\deg(x)}$, with $\|0\|=0$, then we have following:\\
	$\emph(1)~\|{x}\|=0$ if and only if $x=0$;\\
	$\emph(2)~\|{xy}\|=\|{x}\|\cdot\|{y}\|$;\\
	$\emph(3)~\|\alpha x+\beta y\|\leq \max(\|x\|,\|y\|)$ \emph{(for any $\alpha$, $\beta\in\mathbb{K}$)};\\
	$\emph(4)~ $ For any $\alpha, \beta\in\mathbb{K}$, $\alpha\neq 0$, $\beta\neq0$, if $\|x\|\neq\|y\|$, then $\|\alpha x+\beta y\|=\max(\|x\|,\|y\|)$.
	Notice, $P=\{x\in\mathbb{K}((X^{-1})):\|x\|<1\}$ is isomorphic to $\prod_{n\geq1}\mathbb{K}$, which is compact. Thus $\mathbb{K}((X^{-1}))$ is a locally compact non-Archimedean field.
	
\end{example}

\begin{example}
	Let $p$ be a prime number, and $\mathbb{Q}_p$ be the p-adic field. Then for any $x\in \mathbb{Q}_p$, we have $x=\sum_{n=n_0}^{+\infty}a_np^n$, where $a_n\in\{0,1,\cdots,p-1\}$. We denote $ord(x)=\inf\{n\in\mathbb{Z}:a_n\neq0\}$. In particular, $ord(0)=\infty$. And define the norm of $x$ to be $|x|_p=p^{-ord(x)}$, then $|\cdot|_p$ is a non-Archimedean norm. Same, we can get $P=\{x\in \mathbb{Q}_p:|x|_p<1\}$ is compact, moreover, $(\mathbb{Q}_p,|\cdot|_p)$ is a locally compact non-Archimedean field.
	
\end{example}

Now, let $k$ be an integer and $\mathbb{F}^{k}=\mathbb{F}\times\mathbb{F}\times\cdots\times\mathbb{F}$ be the $k$-dimensional vector space over $\mathbb{F}$. The norms on $\mathbb{F}^{k}$ is 
$\|\cdot\|:\mathbb{F}^{k}\to[0,+\infty)$ with $\|{x}\|=\max_{1\leq i\leq {k}}\|{x_i}\|$, where $x=(x_1,x_2,\cdots,x_{k})\in\mathbb{F}^{k}$, this norm is also a non-Archimedean one as $\|x+y\|\le\max\{\|x\|,\|y\|\}$ for any ${x,y}\in\mathbb{F}^{k}$. Without causing misunderstanding, we still let $d$ on $\mathbb{F}^{k}$  be the nature metric induced by $\parallel\cdot\parallel$. Then it is easy to check $(\mathbb{F}^{k},d)$ is complete and locally compact. Besides, let $m_1=\mu\times\mu\cdots\times\mu$ be the product measure on $\mathbb{F}^{k}$, and $m_2$ be the Haar measure on $\mathbb{F}^{k}$. Then for any set $A\in\mathcal{B}(\mathbb{F})\times\mathcal{B}(\mathbb{F})\cdots\times\mathcal{B}(\mathbb{F})$, both of $m_1$ and $m_2$ are satisfied $m_i(x+A)=m_i(A)$, $i=1,2$. Notice $\mathcal{B}(\mathbb{F})\times\mathcal{B}(\mathbb{F})\cdots\times\mathcal{B}(\mathbb{F})$ is a semialgebra, thus $m_1=m_2$ due to the measure extension theorem, we write $m_1=m_2=\mu^{k}$.

\section{Singular value decomposition on $\mathbb{F}^n$}

\subsection{Isometric transformations} Let $\mathbb{F}$ be a non-Archimedean field, recall that $\mathcal{L}(\mathbb{R}^n, \mathbb{R}^n)$ in $\mathbb{R}^n$, we  let  $\mathcal{L}(\mathbb{F}^n, \mathbb{F}^n)$ represent all linear transformations from $\mathbb{F}^n$ to $\mathbb{F}^n$. We choose the base of vector space $\mathbb{F}^n$ as the natural base, then the linear transformations can be regarded as a matrix. Denote $\|T\|=\sup_{\|x\|\neq 0}\frac{\|Tx\|}{\|x\|}$ as the norm of the transformation $T$.

\begin{prop}
For any ${T,A,B}\in \mathcal{L}(\mathbb{F}^n, \mathbb{F}^n)$, we have\\
	\emph{(i)} $\|T\|=\max\limits_{1\leq i,j\leq n}\|T_{ij}\|,$ where $T_{i,j}$ represents the element in row $i$ and column $j$ of $T$.\\
	\emph{(ii)}	$\|AB\|\leq\|A\|\cdot\|B\|$, $\|A+B\|\leq\max\{\|A\|,\|B\|\}$.
\end{prop}
\begin{proof}
	(i) On the one hand, let $T_i$ be the $i-th$ row vector of $T$, for any ${x}\in\mathbb{F}^n$, we have $\|Tx\|\leq\max_i{\|T_i{x}\|}\leq(\max_{i,j}\|T_{i,j}\|)(\max_j\|x_j\|)=(\max_{1\leq i,j\leq n}\|T_{ij}\|)\|x\|.$ On the other hand, without loss of generality, 
	let $\max_{i,j}\|T_{ij}\|=\|T_{i_0,j_0}\|$, $e_{i_0}=(0,\cdots,1,\cdots,0)$, where 1 appears in position $i_0$. Then $$\|{e_{i_0}}\cdot{T}\|=\|(T_{{i_0},1},T_{{i_0},2},\cdot,T_{{i_0},n})\|=\|T_{i_0,j_0}\|.$$ Thus $\|T\|\geq\frac{\|Te_{i_0}\|}{\|e_{i_0}\|}=\|T_{i_0,j_0}\|$.\\
	By using trigonometric inequality and result of (i), we can directly get (ii).
\end{proof}

\begin{remark}
	Let $T\in\mathcal{L}(\mathbb{F}^m, \mathbb{F}^n)$ with $m\neq n$, we still
	denote $$\|T\|=\sup_{\|x\|\neq 0}\frac{\|Tx\|}{\|x\|},$$ where ${x}\in\mathbb{F}^n$. Under this situation, if $A$ and $B$ are matrices that can be added and multiplied, Proposition 3.1 still holds.
\end{remark}

\begin{prop}
Let $T, U$ be two $n\times n$ matrices in the filed $\mathbb{F}$. Then \\
\emph{(i)} $\det(T U)=\det T\det U.$\\
\emph{(ii)} $\det T\neq 0$ if and only if its inverse matrix $T^{-1}$ exists. Moreover, $T^{-1}=\frac{T^{*}}{\det T}$, where $T^{*}$ denotes the adjoint matrix of $T$.\\
\emph{(iii)} $\det T\neq 0$ if and only if the linear equation system $Tx=b$ has a unique solution $x=T^{-1}b$ for any ector $b$.\\
\emph{(iv)} $\|\det T\|\leq\|T\|^n$.
\end{prop}
\begin{proof}
	The first three statements are similar to the case of $\mathbb{R}^n$, which can be extended to a general field. The last statement is directly from the strong triangle 
	inequality.
\end{proof}

\begin{prop}
Let $T\in \mathcal{L}(\mathbb{F}^n, \mathbb{F}^n)$. Then $\|Tx\|=\|x\|$ for any  $ x\in \mathbb{F}^n$ if and only if $\|T\|=\|\det T\|=1.$
\end{prop}
\begin{proof}
	If $\|Tx\|=\|x\|$, for any $x\in\mathbb{F}^n$, then $\|T\|=1$ and $\|\det T\|\leq1=\|T\|^n$. Notice that $Tx=0$ has only one solution $x=0$, thus $\det T\neq0$, and $T^{-1}$ exists. Then for any $x\in\mathbb{F}^n$, we have $\|x\|=\|T\cdot T^{-1}x\|=\|T^{-1}x\|$, so $\|T^{-1}\|=1$ and $\|\det T^{-1}\|\leq1$, which imply $\|\det T\|=\|\det T^{-1}\|=1$. Conversely, we only need to show $\|x\|\leq\|Tx\|$, because $\|Tx\|\leq\|T\|\cdot\|x\|=\|x\|$. By $\|T\|=1$, we have $\|T^{*}\|\leq1$, moreover we get $\|T^{-1}\|=\|\frac{T^{*}}{\det T}\|=\|T^{*}\|\leq1.$ Then for any $x\in\mathbb{F}^n,$ $\|x\|=\|T^{-1}\cdot Tx\|\leq\|Tx\|\cdot\|T^{-1}\|=\|Tx\|.$
\end{proof}

\begin{remark}
	\emph{(i)} For convenience, we denote $$I_n=\{T\in\mathcal{L}(\mathbb{F}^n, \mathbb{F}^n): \|T\|=1=\|\det T\|\},$$ the set $I_n$ will play the role of the orthogonal transformations set in the case of  $\mathbb{R}^n$. We call elements in $I_n$ isometric transformations. For example, the matrix
$\left(\begin{array}{cc }
1-c & -c \\
1 & 1
\end{array}\right)\in I_2$, where $c\in F$ with $\|c\|<1$.\\
	\emph{(ii)} If $T\in I_n$, then $\|\det T\|=1>0$, i.e. $T^{-1}$ exists, thus $\|x\|=\|T\cdot T^{-1}x\|=\|T^{-1}x\|$, for any $x\in\mathbb{F}^n$, then $T^{-1}\in I_n$. \\
	\emph{(iii)} If $A\in I_n$ and $B\in I_n$, then $AB\in I_n$.\\
	\emph{(iv)} For any $P,Q\in I_n$, and any $D\in\mathcal{L}(\mathbb{F}^n, \mathbb{F}^n)$, we have $\|PDQ\|=\|D\|$.
\end{remark}

There are three types of elementary isometrics, which correspond to three types of row operations:\\
(i) Row(column)-multiplying matrices. Multiply all the elements in the i-th row of the identity matrix by $a$ and record that the matrix after multiplication is $T_i(a)$, where $a\in\mathbb{F}$ satisfies $\|a\|=1$.\\
(ii) Row(column)-switching matrices. Exchange the positions of rows i and j of the identity matrix, and record the obtained matrix as $T(i,j)$.\\
(iii) Row(column)-addition matrices. Multiply row $i$ of the identity matrix by $a$ and add it to row $j$, and record the resulting matrix as $T_{ij}(a)$, where $a\in \mathbb{F}$ with $\|a\|\leq 1$.\\
We denote the collection of such matrices by $EI_n(\mathbb{F})$(or just $EI_n$, without causing ambiguity), then, obviously, $EI_n\subset I_n$. More over, $T\in I_n\Leftrightarrow T=T_1\cdot T_2\cdots T_k$, $T_i\in EI_n$, $i\leq k,$ for some $k\in \mathbb{N}$, this is a direct corollary of the Theorem 3.1.

\subsection{Singular value decomposition}

In the following, we give a singular value decomposition(SVD) of matrix in $\mathbb{F}^n$ and obtain the corresponding singuar value, while Kedlaya \cite{25} gave another direct definition of singular value without introducing SVD.

\begin{thm}
Let $T\in\mathcal{L}(\mathbb{F}^n, \mathbb{F}^n)$ be a non-singular matrix, then there
exists a factorization of the form $$T=PDQ,$$ where $P, Q\in
I_n$, and $D=\mbox{diag}(\sigma_1,\sigma_2,\cdots,\sigma_n)$ is a
diagonal matrix, in which $\sigma_i\in \mathbb{F}$,  and $\|\sigma_1\|\geq \|\sigma_2\|\geq
\cdots\geq\|\sigma_n\|>0$. Moreover, the norm of each entry
of the diagonal of $D$ is uniquely determined by $T$.
\end{thm}
\begin{proof}
	Since $T$ is non-singular, there exists $T_{ij}\neq
	0$ such that $\|T_{ij}\|=\|T\|$. By repeatedly multiplying $T$ by row(column)-switching matrices, the entry $T_{ij}$ can be moved to the 
	$(1,1)$ position. Hence without loss of generality we assume that
	$\|T_{11}\|=\|T\|\neq 0$. Then by adding proper multiples of the first row (or column) to the other rows (or columns), $T$
	can be converted into a matrix of the form $$\left(\begin{array}{cc
		}
		T_{11} & 0 \\
		0 & *
	\end{array}\right).$$ It is not hard to verify that the above
	elementary operations can be carried out by left (or right) multiplying $T$ by row(column)-addition matrices. Since $T$ is non-singular, there is at least one non-zero entry in $*$. Repeatedly proceeding this process, the  matrix $T$ will ultimately be converted into a diagonal form. Notice that in this process, we only multiply $T$ by elementary isometric matrices. Hence, exist $P,Q\in I_n$ such that $T$ has the form $T=PDQ,$ where $D$ is a non-singular diagonal matrix. Moreover, we could label  the entries in the diagonal of $D$  in decreasing order of its norm.
	
	Next we  prove the uniqueness. Assume $T=P'D'Q'$, where $P', Q'\in I_p(n)$ and $D'=\text{diag}(\sigma'_1,\sigma'_2,\cdots,\sigma'_n)$ is a diagonal matrix, in which $\sigma'_i\in \mathbb{F}$,  and $\|\sigma'_1\|\geq \|\sigma'_2\|\geq \cdots\geq \|\sigma'_n\|>0$. Next we show $\|\sigma_i\|=\|\sigma'_i\|$, for all $1\leq i\leq n.$
	
	Since $T=PDQ=P'D'Q',$ we get that $$P^{-1}P'D'=DQQ'^{-1}.$$ For convenience, write $U=P^{-}P'$, $V=QQ'^{-1}$, then $U, V\in I_n$ and $UD'=DV$. Since $$\det
	U=\sum_{j_1j_2\cdots j_n}\pm U_{1j_1}U_{2j_2}\cdots U_{nj_n},$$ where
	the summation is taken over all permutations of $(1,2\cdots,n)$, we 
	get that
	$$1=\|\det U\|\leq \max_{j_1j_2\cdots j_n}\|U_{1j_1}U_{2j_2}\cdots U_{nj_n}\|\leq \|U\|^n=1,$$ by using the strong triangle inequality.
	Hence there exists $(j_1,j_2,\cdots,j_n)$, such that
	$\|U_{1j_1}\|=\|U_{2j_2}\|\cdots =\|U_{nj_n}\|=1.$ Since $UD'=DV$,  for each $1\leq k\leq n$, $(UD')_{kj_k}=(DV)_{kj_k}$. This gives that
	$$\|U_{kj_k}\sigma'_{j_k}\|=\|\sigma'_{j_k}\|=\|V_{kj_k}\sigma_k\|\leq \|\sigma_k\|,$$ for any $1\leq k\leq n$, since $V\in I_n$.
	Hence,
	\begin{equation}\label{eq2}
		\|\sigma'_{j_k}\|\leq \|\sigma_k\|,\text{ for any }1\leq k\leq n.
	\end{equation}
	But since $P, Q, P', Q'\in I_p(n)$, we have
	$$\|\sigma_1\sigma_2\cdots\sigma_n\|=\|\sigma'_{j_1}\sigma'_{j_2}\cdots\sigma'_{j_n}\|=\|\det T\|.$$ So none of the inequality in $(\ref{eq2})$ can hold strictly. Combining with the assumption $$\|\sigma'_1\|\geq \|\sigma'_2\|\geq \cdots\geq \|\sigma'_n\|>0,$$ we get that $\|\sigma_i\|=\|\sigma'_i\|$, for all $1\leq i\leq n.$
\end{proof} 

\begin{remark}
	 Let $\alpha_i=\|\sigma_i\|$, $1\leq i\leq n$. Then $\alpha_1\geq \alpha_2\geq\cdots \geq\alpha_n>0$ are uniquely determined by $T$. We call them the \emph{singular values} of  $T$. Obviously, $\alpha_1=\|T\|=\|D\|$, moreover, if $T$ is contractive and non-singular, then we have $1>\|T\|=\alpha_1\geq\alpha_2\geq\cdots\geq\alpha_n>0$.\\

\end{remark}

\subsection{Singular value functions}
\begin{defi}
	Let $T\in\mathcal{L}(\mathbb{F}^n, \mathbb{F}^n)$ be a non-singular matrix, and $s\geq0$, $\lceil s\rceil$ means the minimum integer greater than or equal to $s$. The singular value function of $T$ is
	\begin{equation}
		\phi^s(T)=\left\{
		\begin{array}{rcl}
	\alpha_1\alpha_2\cdots\alpha_{m-1}\alpha_m^{s-m+1}&& {\text{if }0\leq s\leq n,~m=\lceil s\rceil}\\
		(\alpha_1\alpha_2\cdots\alpha_n)^{\frac{s}{n}}=\|\det T\|^{\frac{s}{n}}&&{\text{if }s>n.}	
		\end{array}\right.
	\end{equation}
\end{defi}

 It is clear that as a function of $s$, $\phi^s(T)$ is continuous and strictly decreasing, when $T$ is contractive. Next theorem gives the submultiplicativity of $\phi^s$.

\begin{thm}\label{prop6}
For  $s\geq 0,$ $\phi^s$ is submultiplicative, i.e.,
\begin{equation}\label{eq4}\phi^s(TU)\leq \phi^s(T)\phi^s(U)\end{equation} for any $T, U\in \mathcal{L}(\mathbb{F}^n, \mathbb{F}^n)$.
\end{thm}

Before proving Theorem 3.2, we need the following two lemmas.	

\begin{lemma}\label{lemma1}
Given $1\leq t\leq n$, let $T$ and  $U$ be two $n\times n$ matrices whose
entries are taken from a field $G$. Then for each $t\times t$ submatrix
$(TU)^{j_1,j_2,\cdots, j_t}_{i_1,i_2,\cdots,i_t}$ of $TU$,
\begin{equation}\label{eq3}
\det ((TU)_{i_1,i_2,\cdots, i_t}^{j_1,j_2,\cdots,j_t})=\sum_{k_1,k_2\cdots,k_t}\det (T_{i_1,i_2,\cdots,i_t}^{k_1,k_2,\cdots,k_t})\cdot\det ( U_{k_1,k_2,\cdots,k_t}^{j_1,j_2,\cdots,j_t}),
\end{equation}
where the summation is taken over all possible index strings  $1\leq
k_1<k_2\cdots<k_t\leq n$. 
\end{lemma}
\begin{proof}
	When $t=1$, $(\ref{eq3})$ is exactly the ordinary
	matrix multiplication.
	
	Suppose by induction that $(\ref{eq3})$ holds for $t$. We now prove
	$(\ref{eq3})$ for $t+1$.
	
	For convenience of marking, we assume that the $(t+1)\times
	(t+1)$ submatrix  in the left side of $(\ref{eq3})$ is
	$(TU)_{1,\cdots,t+1}^{1,\cdots,t+1}$. Then by expanding its
	determinant along the first row, we get
	$$
	\det ((TU)_{1,\cdots,t+1}^{1,\cdots,t+1})=\sum_{j=1}^{t+1}(-1)^{1+j}(TU)_{1,j}\det (TU)_{2,\cdots,t+1}^{1,\cdots,j-1,j+1,\cdots,t+1}.
	$$
	By using the induction assumption, it follows that
	\begin{eqnarray*}
		&&\det ((TU)_{1,\cdots,t+1}^{1,\cdots,t+1})\\
		&=&\sum_{j=1}^{t+1}\left((-1)^{1+j}\sum_{l_1=1}^n T_{1l_1}U_{l_1j}\cdot\sum_{l_2,\cdots ,l_{t+1}}\det T_{2,\cdots,t+1}^{l_2,\cdots,l_{t+1}}\cdot\det U_{l_2,\cdots,l_{t+1}}^{1,\cdots,j-1,j+1,\cdots,t+1}\right)\\
		&=&\sum_{j=1}^{t+1}\left((-1)^{1+j}\sum_{\tiny\begin{array}{ccc} &1\leq l_1\leq n, \\&l_2,\cdots,l_{t+1}\end{array}}T_{1l_1}\det T_{2,\cdots,t+1}^{l_2,\cdots,l_{t+1}}U_{l_1j}\det U_{l_2,\cdots,l_{t+1}}^{1,\cdots,j-1,j+1,\cdots,t+1}\right)\\
		&=&\sum_{\tiny\begin{array}{ccc} &1\leq l_1\leq n,\\ &l_2,\cdots,l_{t+1}\end{array}}\left(T_{1l_1}\det T_{2,\cdots,t+1}^{l_2,\cdots,l_{t+1}}\sum_{j=1}^{t+1}(-1)^{1+j}U_{l_1j}\det U_{l_2,\cdots,l_{t+1}}^{1,\cdots,j-1,j+1,\cdots,t+1}\right).
	\end{eqnarray*}
	
	For fixed $1\leq l_1\leq n$, $1\leq l_2<\cdots<l_{t+1}\leq n$,
	consider the term $$\sum_{j=1}^{t+1}(-1)^{1+j}U_{l_1j}\det
	U_{l_2,\cdots,l_{t+1}}^{1,\cdots,j-1,j+1,\cdots,t+1}.$$ Notice that
	if $l_1\in\{l_k: 2\leq k\leq t+1\}$,  the above term is equal
	to $0$. Otherwise if $l_1\notin\{l_k: 2\leq k\leq t+1\}$, we have
	\begin{eqnarray*}&&\sum_{j=1}^{t+1}(-1)^{1+j}U_{l_1j}\det U_{l_2,\cdots,l_{t+1}}^{1,\cdots,j-1,j+1,\cdots,t+1}\\&=&(-1)^{1+\widetilde{l_1}}\sum_{j=1}^{t+1}(-1)^{\widetilde{l_1}+j}U_{l_1j}\det U_{l_2,\cdots,l_{t+1}}^{1,\cdots,j-1,j+1,\cdots,t+1}\\&=&(-1)^{1+\widetilde{l_1}}\det U^{1,\cdots,t+1}_{l_1\vee\{l_2,\cdots,l_{t+1}\}},\end{eqnarray*}
	where $l_1\vee\{l_2,\cdots,l_{t+1}\}$ denote the index string
	consisting of $l_1,l_2,\cdots,l_{t+1}$ in increasing order, and
	$\widetilde{l_1}=\#\{k: l_k\leq l_1\}$.

	Hence, we have \begin{eqnarray*}
		&&\det ((TU)_{1,\cdots,t+1}^{1,\cdots,t+1})\\
		&=&\sum_{\tiny\begin{array}{ccc}&l_1\vee\{l_2,\cdots,l_{t+1}\}, \\& l_1\notin\{ l_2,\cdots,l_{t+1}\}\end{array}}(-1)^{1+\widetilde{l_1}}T_{1l_1}\det T_{2,\cdots,t+1}^{l_2,\cdots,l_{t+1}}\cdot\det U^{1,\cdots,t+1}_{l_1\vee\{l_2,\cdots,l_{t+1}\}}\\
		&=&\sum_{k_1,\cdots,k_{t+1}}\left(\det U^{1,\cdots,t+1}_{k_1,\cdots,k_{t+1}}\cdot\sum_{\tiny\begin{array}{ccc} &(l_1\vee\{l_2,\cdots,l_{t+1}\})\\&=(k_1,\cdots,k_{t+1})\end{array}}(-1)^{1+\widetilde{l_1}}T_{1l_1}\det T_{2,\cdots,t+1}^{l_2,\cdots,l_{t+1}}\right)\\
		&=&\sum_{k_1,\cdots,k_{t+1}}\det T_{1,\cdots,t+1}^{k_1,\cdots,k_{t+1}}\cdot\det U^{1,\cdots,t+1}_{k_1,\cdots,k_{t+1}}.
	\end{eqnarray*}
	
	So the lemma is complete.
\end{proof}  

\begin{lemma}\label{lemma2}
Let $1\leq s\leq n$ be an integer. Then, for any non-singular $T\in\mathcal{L}(\mathbb{F}^n,\mathbb{F}^n)$, we have 
$$\phi^s(T)=\max_{T_s}\{\|\det T_s\|\},$$ where the  maximum is
taken over all $s\times s$ submatrices of $T$.

\end{lemma}
\begin{proof}
	First, there are $P, Q\in I_n$, such that $T=PDQ$ with $D=\mbox{diag}(\sigma_1,\sigma_2,\cdots,\sigma_n)$ is a diagonal matrix, $\sigma_i\in \mathbb{F}$, and $\|\sigma_1\|\geq \|\sigma_2\|\geq \cdots\geq\|\sigma_n\|>0$.
	
	Let $T_s=T_{i_1,\cdots,i_s}^{j_1,\cdots,j_s}$ be any $s\times s$
	submatrix of $T$. Then by Lemma $\ref{lemma1}$, we have
	\begin{eqnarray*}
		\det T_s&=&\sum_{\tiny\begin{array}{ccc}&k_1,\cdots,k_s, \\& l_1,\cdots,\l_s\end{array}}\det P_{i_1,\cdots,i_s}^{k_1,\cdots,k_s}\cdot\det D_{k_1,\cdots,k_s}^{l_1,\cdots,l_s}\cdot\det Q_{l_1,\cdots,l_s}^{j_1,\cdots,j_s}\\
		&=&\sum_{k_1,\cdots,k_s}\det P_{i_1,\cdots,i_s}^{k_1,\cdots,k_s}\cdot\det D_{k_1,\cdots,k_s}^{k_1,\cdots,k_s}\cdot\det Q_{k_1,\cdots,k_s}^{j_1,\cdots,j_s}.
	\end{eqnarray*}
	So by the strong triangle inequality, we have $$\|\det T_s\|\leq
	\max_{k_1,\cdots,k_{s}}\|\det
	P_{i_1,\cdots,i_s}^{k_1,\cdots,k_s}\|\cdot\|\sigma_{k_1}\cdots\sigma_{k_s}\|\cdot\|\det
	Q_{k_1,\cdots,k_s}^{j_1,\dots,j_s}\|.$$ Since $P,Q\in I_n$, for each index string $\{k_1,\cdots,k_s\}$, $\|\det
	P_{i_1,\cdots,i_s}^{k_1,\cdots,k_s}\|\leq 1$ and $\|\det
	Q_{k_1,\cdots,k_s}^{j_1,\dots,j_s}\|\leq 1$. Hence $$\|\det
	T_s\|\leq
	\max_{k_1,\cdots,k_{s}}\|\sigma_{k_1}\cdots\sigma_{k_s}\|\leq
	\alpha_1\cdots\alpha_s= \phi^s(T).$$ By the arbitrariness of $T_s$,
	it follows that $\phi^s(T)\geq\max_{T_s}\{\|\det T_s\|\}$.
	
	Next we show that there must exist a submatrix $T_s$ such that $\phi^s(T)=\|\det T_s\|$. Otherwise, if $\phi^s(T)>\|\det T_s\|$ for every submatrix $T_s$. By
	$D=P^{-1}TQ^{-1}$, and $P^{-1}, Q^{-1}\in I_n$, we get
	\begin{eqnarray*}
		\sigma_1\cdots\sigma_s&=&\det((P^{-1}TQ^{-1})_{1,\cdots,s}^{1,\cdots,s})\\&=&\sum_{\tiny\begin{array}{ccc} &k_1,\cdots,k_s, \\&l_1,\cdots,l_s\end{array}}\det (P^{-1})_{1,\cdots,s}^{k_1,\cdots,k_s}\det T_{k_1,\cdots,k_s}^{l_1,\cdots,l_s}\det(Q^{-1})_{l_1,\cdots,l_s}^{1,\cdots,s}.
	\end{eqnarray*}
	Hence
	\begin{eqnarray*}
		\alpha_1\cdots\alpha_s&\leq&\max_{\tiny\begin{array}{ccc}&k_1,\cdots,k_s, \\&l_1,\cdots,l_s\end{array}}\|\det (P^{-1})_{1,\cdots,s}^{k_1,\cdots,k_s}\|\cdot \|\det T_{k_1,\cdots,k_s}^{l_1,\cdots,l_s}\|\cdot \|\det(Q^{-1})_{l_1,\cdots,l_s}^{1,\cdots,s}\|\\&\leq&\max_{\tiny\begin{array}{ccc}&k_1,\cdots,k_s, \\&l_1,\cdots,l_s\end{array}}\|\det T_{k_1,\cdots,k_s}^{l_1,\cdots,l_s}\|\\&<&\phi^s(T)=\alpha_1\cdots\alpha_s.
	\end{eqnarray*} That is, $\alpha_1\cdots\alpha_s<\alpha_1\cdots\alpha_s$, which is a contradiction. It follows that $\phi^s(T)=\max\limits_{T_s}\{\|\det T_s\|\}$.
\end{proof} 

\textit{Proof of Theorem \ref{prop6}.} 
If $1\leq s\leq n$ and $s$ is an
integer. Take any $s\times s$ submatrix $(TU)_s$ of $TU$. Without
loss of generality, we may choose
$(TU)_s=(TU)_{1,\cdots,s}^{1,\cdots,s}$. Then 
$$\det(TU)_{1,\cdots,s}^{1,\cdots,s}=\sum_{k_1,\cdots,k_s}\det T_{1,\cdots,s}^{k_1,\cdots,k_s}\cdot\det U_{k_1,\cdots,k_s}^{1,\cdots,s}.$$
Moreover, we get
$$\|\det(TU)_s\|\leq \max_{k_1,\cdots,k_s}\|\det T_{1,\cdots,s}^{k_1,\cdots,k_s}\|\cdot \|\det U_{k_1,\cdots,k_s}^{1,\cdots,s}\|\leq \phi^{s}(T)\phi^s(U).$$
Since $(TU)_s$ is an arbitrary $s\times s$ submatrix of $TU$, we
have
$$\max_{(TU)_s}\|\det (TU)_s\|\leq \phi^s(T)\phi^s(U),$$ i.e.$$\phi^s(TU)\leq \phi^s(T)\phi^s(U).$$ 
If $1\leq s\leq n$ and $s$ not an integer, let $m=\lceil s \rceil $, then
\begin{eqnarray*}
\phi^s(T)=\alpha_1\cdots\alpha_{m-1}\alpha_m^{s-m+1}&=&(\alpha_1\cdots\alpha_{m-1}\alpha_m)^{s-m+1}(\alpha_1\cdots\alpha_{m-1})^{m-s}\\&=&
(\phi^m(T))^{s-m+1}(\phi^{m-1}(T))^{m-s}.
\end{eqnarray*}
it follows that $$\phi^s(TU)=(\phi^m(TU))^{s-m+1}(\phi^{m-1}(TU))^{m-s}\leq\phi^s(T)\phi^s(U).$$
As for $s\geq n$, it is obvious that $\phi^s$ is multiplicative. \hfill $\Box$\\

\section{ self-affine sets in $\mathbb{F}^n$}
 
\subsection{Notations}
Firstly, we give some notations. Let $M\geq2$ be an integer. Denote $$J_{\infty}=\{1,2,\cdots,M\}^{\mathbb{N}},~J_r=\{(i_1,\cdots,i_r):  1\leq i_j\leq M,~ j\leq r\},$$ $J=\bigcup_{r\geq0}J_r$, and $J_0=\emptyset$. For any $w\in J$, $v\in J$ or $J_{\infty}$, denote the length of $w$ by $|w|$, and denote $$wv$$ to be the sequence obtained by juxtaposition of the terms of $w$ and $v$. If $v=ww'$, for some $w'\in J_{\infty}$, we write $$w<v.$$ If $w$ and $v\in J_{\infty}$, then $$w\wedge v$$ is the maximal sequence such that both $w\wedge v<w$
and $w\wedge v<v$. Let $$d(w,v)=2^{-|w\wedge v|},$$ for $w\neq v\in J_{\infty}$. Then $(J_{\infty},d)$ is a compact metric space. If $w=(i_1,i_2,\cdots,i_k)\in J_k,$ we denote  $$T_w=T_{i_1}\cdot T_{i_2}\cdots T_{i_k}.$$

\begin{prop}\label{p2}
	Let $\{T_1,T_2,\cdots,T_M\}$ be sequence of contractive non-singular linear transformations on $\mathbb{F}^n$.\\
	\emph{(i)} For every two finite words $w, w'\in J,$ $\phi^s(T_{ww'})\leq \phi^s(T_w)\phi^s(T_{w'}).$ \\
	\emph{(ii)} There exist $b\leq a<1$ s.t. $b^{s|w|}\leq\phi^s(T_w)\leq a^{s|w|}$, for any $w\in J$.
\end{prop}
\textit{Proof.} By using submultiplicativity of $\phi^{s}$, we can get (i). Because $T_i$ is contractive, there exists $b\leq a<1$, s.t. 
$1>a\geq\alpha^{(i)}_1\geq\cdots\geq\alpha^{(i)}_n\geq b>0$, for any $i\leq M$, where $\alpha^{(i)}_j$ is the $j$-th singular value of $T_i$. Then, for any $w\in J$, it is obviously that $\phi^s(T_w)\leq a^{s|w|}$. Notice, $\frac{1}{\alpha^{(i)}_j}$ is the j-th singular value of $T^{-1}_i$, thus we have $$\frac{1}{b}\geq\frac{1}{\alpha^{(i)}_n}\geq\cdots\frac{1}{\alpha^{(i)}_1}\geq\frac{1}{a}>1.$$ It follows that  $\phi^s(T^{-1}_i)\leq |\frac{1}{b}|^s$, for any $i\leq M.$ Let $w=(w_1,\cdots,w_l)$, then $$1=\phi^s(I)\leq\phi^s(T_w)\phi^s(T^{-1}_{w_1})\cdots\phi^s(T^{-1}_{w_l})\leq(\frac{1}{b})^{s|w|}\phi^s(T_w).$$ Then we have (ii). \hfill $\Box$\\

\subsection{Self-affine set}
Write $\{S_i=T_i+b_i\}_{i\leq M}$ as an \textit{affine
	iterated function system}(\emph{AIFS}) on $\mathbb{F}^n$, and let $K(\textbf{b})$ be the associated invariant set(\textit{self-affine set}), where $\textbf{b}=(b_1,b_2,\cdots,b_M)\in \mathbb{F}^{nM}$. The existence of $K(\textbf{b})$ follows from the completeness of the metric space consisting of all compact subsets of $\mathbb{F}^n$ with the Hausdorff metric. Actually, we have classical results about invariant set existence in general metric space, see \cite{26} Theorem 2.5.3.
\begin{lemma}
Let $(X,d)$ be a complete metric space, $\mathcal{H}(X)=\{E: E$ is compact subset of $X\}$. If $d_h$ is Hausdorff metric in $X$,  then $(\mathcal{H}(X),d_h)$ is complete.
\end{lemma}
So the only thing we need to verify is that $S(E)=\bigcup_{i\leq M}S_i(E)$ is contraction map on $(\mathcal{H}(\mathbb{F}^n),d_h)$. Actually, for any $A$, $B\in \mathcal{H}(\mathbb{F}^n)$, recall $d_h(A,B)=\max\{d(A,B),d(B,A)\}$, and $d(A,B)=\sup_{x\in A}d(x,B)$, then we have
\begin{align*}
	d_h(S(A),S(B))&=d_h(\bigcup_{i\leq M}S_i(A),\bigcup_{i\leq M}S_i(B))\\
	&\leq\max(\max_{1\leq i\leq M}d(S_i(A),S_i(B)),\max_{1\leq i\leq M}d(S_i(B),S_i(A)))\\
	&=\max(\max_{1\leq i\leq M}d(T_i(A),T_i(B)),\max_{1\leq i\leq M}d(T_i(B),T_i(A)))\\
	&\leq (\max_{1\leq i\leq M}\|T_i\|)\cdot d_h(A,B). 
\end{align*}
Thus, by the fixed point theorem, there is a unique $K(\textbf{b})\in\mathcal{H}(X)$ such that 
$S(K(\textbf{b}))=K(\textbf{b})$.
 
Moreover, if $\omega=\omega_1\omega_2\cdots\in J_{\infty}$, write
\begin{eqnarray*}
	x_\omega(\textbf{b})&=&\lim_{m\rightarrow+\infty}(T_{\omega_1}+b_{\omega_1})(T_{\omega_2}+b_{\omega_2})\cdots(T_{\omega_m}+b_{\omega_m})(0)\\
	&=&b_{\omega_1}+T_{\omega_1}b_{\omega_2}+T_{\omega_1}T_{\omega_2}b_{\omega_3}+\cdots,
\end{eqnarray*}
which is well-defined since all the $T_i$'s are contractive. Then, we have $\bigcup_{\omega\in J_{\infty}}x_\omega(\textbf{b})=K(\textbf{b})$.

\subsection{Hausdorff dimension}
For $s\geq 0$, the $s$-dimensional Hausdorff measure  
$\mathcal{H}_d^s$ is defined on subsets of $(X,d)$ by 
$$\mathcal{H}^s(A)=\lim_{\delta\rightarrow 0}\mathcal{H}_{\delta}^s(A),$$
with 
$$\mathcal{H}_{\delta}^s(A)=\text{inf}\left\{\sum_i(\text{diam} U_i)^s: A\subset \bigcup_i U_i,~ \text{diam} U_i<\delta\right\},$$
where $\{U_i\}$ are open subsets in $X$.
The
\textit{Hausdorff dimension} of $A\subset X$ is
$$\dim_H(A)=\inf\{s:\mathcal{H}^s(A)=0\}=\sup\{s:\mathcal{H}^s(A)=+\infty\}.$$
See details in \cite{3,5,8}.

\subsection{net measure on $J_{\infty}$}

	Denote $$[\omega]=\{\nu\in J_{\infty}:\omega<\nu\},$$ for any $\omega\in J$. For a finite set $A\subset J$, set $A$ is called  the $r$\text{-cover set of }$E(\subset J_{\infty})$ if $E\subset\bigcup_{w\in A}[w]$ and $|w|>r$ for any $w\in A$. Let $$\mathcal{M}^s_r(E)=\inf \{\sum_{w\in A}\phi^s(T_w):~A\text{ is r-cover set of }E\},$$ and $\mathcal{M}^s(E)=\lim_{r\rightarrow +\infty}\mathcal{M}^s_r(E).$ Then we have following results.
\begin{thm}\label{t1}
	 There exists a unique $s>0$, s.t. $\lim_{r\rightarrow +\infty}\frac{1}{r}\log(\sum_{w\in J_r}\phi^{s}(T_w))=0$, we denote the number $s$ as $d(T_1,\cdots,T_M)$. Moreover, we have the following two equivalent definitions:\\
	 \emph{(i)} $s=\inf\{s:\mathcal{M}^s(J_{\infty})=0\}=\sup\{s:\mathcal{M}^s(J_{\infty})=\infty\},$\\ \emph{(ii)} $s=\inf\{s:\sum_{w\in J}\phi^s(T_w)<\infty\}=\sup\{s:\sum_{w\in J}\phi^s(T_w)=\infty\}.$
\end{thm}
\begin{proof}
	The proof of the Theorem 4.1. is the same as Proposition 4.1 in \cite{4}, essentially all that we used was that the $\phi^s(T_i)$ are decreasing in s and its subadditivity, which we already have. Thus, the theorem still holds in $\mathbb{F}$.
\end{proof}  

Further, we have the following Lemma, see [9, Lemma 4.2]
\begin{lemma}\label{la1}
	If $\mathcal{M}^s(J_{\infty})=\infty$ for some $s$. Then there exists a compact set $E\subset J_{\infty}$ such that $0<\mathcal{M}^s(E)<\infty $ and a constant $c_1$ s.t. $$\mathcal{M}^s(E\cap [\omega])\leq c_1\phi^s(T_w) ~~(w\in J).$$
\end{lemma}

\section{ calculation of dimension}
\textit{Proof Theorem \ref{thm1} (i)}. 
Let $d(T_1,\cdots,T_M)<s<n$, $s\notin\mathbb{Z}$, and $m$ be the smallest integer greater than or equal to $s$. Then, by the difintion of $d(T_1,\cdots,T_M)$, we have $$\lim_{r\rightarrow +\infty}\frac{1}{r}\log(\sum_{w\in J_r}\phi^{s}(T_w))<0,$$ i.e. there exists $k$ s.t. $$\sum_{w\in J_k}\phi^{s}(T_w)\leq 1.$$ Fix $\epsilon>0$, and for any $w_0=(i_1,i_2,\cdots)\in J_{\infty}$, let $$w=(i_1,\cdots,i_{pk})<w_0.$$ We choose $p$($p>0$) to be the smallest number such that $\epsilon\geq\alpha_m>b^k\epsilon$, where $\alpha_m$ is the $m$-th singular value of $T_{w}$ and $b$ is defined in Propostion 4.1. Using subadditivity repeatedly, we have $$\sum_{w\in A}\phi^{s}(T_{w})\leq 1,$$ where $A=\{w:w_0\in J_{\infty}\}$ is a cover set of $J_{\infty}$, i.e. $J_{\infty}\subset\bigcup_{w\in A}[w]$. 
Notice that there exists large enough $R>1$ such that $S_i(B)\subset B$, for any $1\leq i\leq M,$ where $B=\{x\in\mathbb{F}^n:\|x\|\leq R\}$. Thus, we have $$K(\textbf{b})=\bigcup_{w\in J_{\infty}}S_w(B)\subset\bigcup_{w\in A}S_w(B).$$ 
Fix a $w\in A$, by $S_w(B)=T_w(B)+c$, where $c=(c_1,\cdots,c_n)\in \mathbb{F}^n$, we have $$S_w(B)=\{x\in\mathbb{F}^n:\|x_i-c_i\|\leq \alpha_i R,~1\leq i\leq n\},$$ where $\alpha_i$ is the $i$-singular value of $T_w$.
For any $1\leq i<m$, let $$E_i=\{x\in\mathbb{F}^n:\|x-c_i\|\leq \alpha_i R\}=c_i+\sigma_i\cdot y\cdot O,$$ where $\sigma_i,~y\in\mathbb{F}$, s.t. $\|y\|=R$ and $\|\sigma_i\|=\alpha_i$(this $y$ can be found, otherwise we can make $R$ larger).
Notice, $$O=\sum_{d_1\in R}(d_1+\pi O)=\sum_{d_1,d_2\in R}(d_1+\pi d_2+\pi^2O)=\cdots,$$ where $\sum$ means disjiont union of set and $R$ is a finite set of representatives for $O/P$ with ${\rm card}(R)=q=\frac{1}{\|\pi\|}$.
Thus $$E_i=\sum_{d_1\in R}(c_i+\sigma_iyd_1+\sigma_iy\pi O)=\sum_{d_1,d_2\in R}(c_i+\sigma_iyd_1+\sigma_iy\pi d_2+\sigma_iy\pi^2 O)=\cdots~(i<m).$$
That is, $E_i$ can be covered by $q^t$ nonintersecting "balls" of diameter $\frac{\alpha_iR}{q^t}$, i.e. an interval with length $\frac{\alpha_iR}{q^t}$.
Then, for any $1\leq i<m$, exists $t\in\mathbb{N}$ s.t. $$q^{t-1}<\frac{\alpha_i}{\alpha_m}\leq q^t,$$ then $$\frac{\alpha_iR}{q^t}\leq\alpha_mR<\frac{\alpha_iR}{q^{t-1}}.$$ 
So, $E_i$ can be covered by $\frac{\alpha_iq}{\alpha_m}$ intervals of diameter $\alpha_mR$. Moreover, $S_w(B)$ can be covered by $$\frac{\alpha_1\cdots\alpha_{m-1}q^{m-1}}{\alpha_m^{m-1}}$$ cubes with side length of $\alpha_mR$.
Then we have $$N_{\alpha_mR}(K(\textbf{b}))\leq\sum_{w\in A}\phi^s(T_{w})\alpha_m^{-s}q^{m-1}\leq b^{-ks}\epsilon^{-s}q^{m-1},$$ where $N_{\delta}(A)$ is the number of cubes of side $\delta$ in a cubical lattice that cover the set $A$. Thus, we have $$\overline{\dim}_BK(\textbf{b})\leq\overline{\lim}_{\epsilon\rightarrow0}\frac{\log N_{\alpha_mR}(K(\textbf{b}))}{-\log\alpha_mR}\leq \overline{\lim}_{\epsilon\rightarrow0}(c+s\log\epsilon)/\log\epsilon\leq s,$$
since $\epsilon\geq\alpha_m>b^k\epsilon$. It follows that $\dim_HK(\textbf{b})\leq \overline{\dim}_BK(\textbf{b})\leq d(T_1, T_2,\cdots,T_M)$ for all $\textbf{b}\in \mathbb{F}^{nM}$.\hfill$\Box$\\

Before proving Theorem 1.1 (ii), we need the following four lemmas.
\begin{lemma}\label{lemma3}
Let $r>0$ and exists $\alpha\in\mathbb{F}$, such that $\|\alpha\|=r$. Then we have the following estimates.
$$\int_{B^n_r}\|z\|^{-s}d\mu^n(z)=O(r^{n-s}), \mbox{ as } r\rightarrow 0, \mbox{ for } 0<s<n,$$
and
$$\int_{\mathbb{F}^n\setminus B^n_{r}}\|z\|^{-s}d\mu^n(z)=O(r^{n-s}), \mbox{ as }r\rightarrow 0, \mbox{ for } s>n,$$
where $B^n_r=\{x\in\mathbb{F}^n:\|x\|\leq r\}$.
\end{lemma}

\begin{proof}
	Recall, $\pi O=P$, $\|\pi\|=q^{-1}$, then there exists $l\in\mathbb{Z}$  s.t. $q^{-(l+1)}\leq r<q^{-l}$. Let $0<s<n$, then
	\begin{eqnarray*}
		\int_{B^n_{r}}\|z\|^{-s}d\mu^n(z)&\leq&\sum_{i=l}^{+\infty}\int_{B^n_{q^{-i}}\setminus B^n_{q^{-i-1}}}\|z\|^{-s}d\mu^n(z)\\&=&\sum_{i=l}^{+\infty}q^{(i+1)s}(q^{-in}-q^{-(i+1)n})\\&=&\frac{q^{(s-n)l}}{1-q^{s-n}}q^s(1-q^{-n})\leq C\cdot r^{n-s}. 
	\end{eqnarray*}
	When $s>n$, similarly, the following formula can be obtained
	\begin{eqnarray*}
		\int_{\mathbb{F}^n\setminus B^n_{r}}\|z\|^{-s}d\mu^n(z)&=&O(r^{n-s}). 
	\end{eqnarray*}
	Thus, we got the proof.
\end{proof}

\begin{lemma}\label{prop7}
Let $0<s<n$ be a non-integral real number. Then there exists a constant $c$ depending only on $n$, $s$ and $r$ such that
$$\int_{B^n_r}\frac{d\mu^n(x)}{\|Tx\|^s}\leq \frac{c}{\phi^s(T)}$$ for all non-singular $T\in \mathcal{L}(\mathbb{F}^n,\mathbb{F}^n).$
\end{lemma}

\textit{Proof.} Let $T=PDQ$ be a
singular value decomposition of $T$, where
$D=\mbox{diag}(\sigma_1,\sigma_2,\cdots,\sigma_n)$, with
$\alpha_1\geq \alpha_2\geq\cdots \geq\alpha_n>0$  the associated
singular values.

Write $$I=\int_{B^n_r}\frac{d\mu^n(x)}{\|Tx\|^s},$$ then we have
\begin{eqnarray*}
I=\int_{B^n_r}\frac{d\mu^n(x)}{\|DQx\|^s}.
\end{eqnarray*}
Letting $y=Qx$, then $d\mu^n(y)=\|\det Q\|d\mu^n(x)=d\mu^n(x)$ and $\|y\|=\|Qx\|\leq r$, so we obtain
$$I=\int_{B^n_{r}}\frac{d\mu^n(y)}{\|Dy\|^s}=\int_{B^n_{r}}\frac{d\mu^n(y)}{\max_{1\leq i\leq n}\{\|\sigma_iy_i\|^s\}}.$$
Let $z_i=\sigma_iy_id$, where $\|c\|=r^{-1}$, and $c\in\mathbb{F}$, then
$$I=r^n(\alpha_1\cdots\alpha_n)^{-1}\int_{B_{\alpha_1}\times\cdots\times B_{\alpha_n}}\frac{d\mu(z_1)\cdots d\mu(z_n)}{\max_{1\leq i\leq n}\{\|d\mu(z_i)\|^s\}}.$$
Let $m=\lceil s\rceil$. Writing
$$P_1=\{z\in B_{\alpha_1}\times\cdots\times B_{\alpha_n}: \max\{\|z_1\|,\cdots,\|z_m\|\}\leq \alpha_m\}$$
and
$$P_2=\{z\in B_{\alpha_1}\times\cdots\times B_{\alpha_n}: \max\{\|z_1\|,\cdots,\|z_{m-1}\|\}>\alpha_m\},$$
we have $B_{\alpha_1}\times\cdots\times B_{\alpha_n}\subset P_1\cup P_2,$ since $\|z_m\|\leq \alpha_m.$
Let $$I_0=\int_{B_{\alpha_1}\times\cdots\times B_{\alpha_n}}\frac{d\mu(z_1)\cdots d\mu(z_n)}{\max_{1\leq i\leq n}\{\|z_i\|^s\}}.$$ Then
\begin{eqnarray*}
I_0&\leq &\int_{P_1}\frac{d\mu(z_1)\cdots d\mu(z_n)}{\max_{1\leq i\leq m}\{\|z_i\|^s\}}+\int_{P_2}\frac{d\mu(z_1)\cdots d\mu(z_n)}{\max_{1\leq i\leq m-1}\{\|z_i\|^s\}}
\\&\leq& \alpha_{m+1}\cdots\alpha_n\int_{\max\{\|z_1\|,\cdots,\|z_m\|\}\leq \alpha_m}\frac{d\mu(z_1)\cdots d\mu(z_m)}{\max_{1\leq i\leq m}\{\|z_i\|^s\}}\\&&+\alpha_{m}\cdots\alpha_n\int_{\max\{\|z_1\|,\cdots,\|z_{m-1}\|\}>\alpha_m}\frac{d\mu(z_1)\cdots d\mu(z_{m-1})}{\max_{1\leq i\leq m-1}\{\|z_i\|^s\}}\\
&=&\alpha_{m+1}\cdots\alpha_n\int_{B_{\alpha_m}^m}\frac{d\mu(z_1)\cdots d\mu(z_m)}{\|z^{(m)}\|^s}+\\&&\alpha_m\cdots\alpha_n\int_{F^{m-1}\setminus B_{\alpha_m}^{m-1}}\frac{d\mu(z_1)\cdots d\mu(z_{m-1})}{\|z^{(m-1)}\|^s}.
\end{eqnarray*}

By Lemma $\ref{lemma3}$, we get
$$I_0\leq c_1\alpha_{m+1}\cdots\alpha_n \alpha^{(m-s)}_m+c_2\alpha_m\cdots\alpha_n \alpha^{(m-1-s)}_m$$
for appropriate constants $c_1$, $c_2$ depending only on $n$, $s$ and $r$. Hence, $$I\leq (c_1+c_2)r^n\frac{1}{\alpha_1\cdots\alpha_{m-1}\alpha_m^{s-m+1}}=\frac{c}{\phi^s(T)}.$$
The proof is finished.\hfill$\Box$\\

\begin{lemma}\label{lemma4}
If $s\notin\mathbb{Z}$ and $0<s<n$ and $\|T_i\|<1$, for any $1\leq i\leq M$, then there is a number $0<c<+\infty$ such that
$$\int_{\textbf{b}\in B_{r}^{nM}}\frac{d\mu^{nM}(\textbf{b})}{\|x_\omega(\textbf{b})-x_{\omega'}(\textbf{b})\|^s}\leq\frac{c}{\phi^s(T_{\omega\wedge\omega'})}$$
for all distinct $\omega, \omega'\in J_{\infty}.$
\end{lemma}
\begin{proof}
	This is non-Archimedean case of Lemma 3.1 from Falconer \cite{4}. Write $\omega=(\omega\wedge\omega')\nu$ and $\omega'=(\omega\wedge\omega')\nu'$, where $\nu, \nu'\in J_{\infty}.$ Let $q=|\omega\wedge\omega'|$, $$\eta=\max_{1\leq i\leq M}\|T_i\|<1.$$ Notice, $\nu_1\neq\nu'_1$, without loss of generality we may assume that $\nu_1=1$ and $\nu'_1=2.$ Then
	\begin{eqnarray*}
		x_{\nu}(\textbf{b})-x_{\nu'}(\textbf{b})&=&b_1-b_2+(T_{\omega_{q+1}}b_{\omega_{q+2}}+T_{\omega_{q+1}}T_{\omega_{q+2}}b_{\omega_{q+3}}+\cdots)\\
		&&-(T_{\omega'_{q+1}}b_{\omega'_{q+2}}+T_{\omega'_{q+1}}T_{\omega'_{q+2}}b_{\omega'_{q+3}}+\cdots)
		\\&=&b_1-b_2+E(\textbf{b}),
	\end{eqnarray*}
	where $E\in \mathcal{L}(\mathbb{F}^{nM}, \mathbb{F}^n)$, Falconer's proof goes through if $\|E\|<1$. Actually, by non-Archimedean properties of norm, we have $\|E\|=\sup\frac{\|E(\textbf{b})\|}{\|\textbf{b}\|}\leq \sup\frac{\eta\|\textbf{b}\|}{\|\textbf{b}\|}=\eta<1.$
\end{proof}

\begin{lemma}\label{lemma6}
	Let $h$ be a Borel measure on $J_{\infty}$, with $0<h(J_{\infty})<\infty$. If there exists $s<n$ s.t. $$\int_{J_{\infty}}\int_{J_{\infty}}\int_{\textbf{b}\in B_{r}^{nM}}\frac{d\mu^{nM}(\textbf{b})dh(\omega)dh(\omega')}{\|x_\omega(\textbf{b})-x_{\omega'}(\textbf{b})\|^s}\leq\infty,$$ then, for $\mu^{nM}$-a.e. $\textbf{b}\in B^{nM}_r$, $\dim_HK(\textbf{b})\geq s$.
\end{lemma}
\begin{proof}
	The proof of Lemma 5.4 is basically the same as the proof of the Lemma 5.2 in \cite{4}. So we skip the proof.
\end{proof}

\textit{Proof Theorem \ref{thm1} (ii)}. 
 First, we proof that for $\mu^{nM}$-a.e. $\textbf{b}\in B^{nM}_r$, $\dim_HK(\textbf{b})\geq \min\{n,d(T_1,\cdots,T_M).$ 
 Fix $R>0$, let $t\notin\mathbb{Z}$ such that $0<t<\min\{n,d(T_1,\cdots,T_M)\}$, and choose $s$ such that $t<s<\min\{n,d(T_1,\cdots,T_M)$. Then, we have  $\mathcal{M}^s(J_{\infty})=\infty$, by using Lemma \ref{la1}, there exists a compact set $E\subset J_{\infty}$ such that $0<\mathcal{M}^s(E)<\infty $, and exists a constant $c_1$ such that $$\mathcal{M}^s(E\cap [\omega])\leq c_1\phi^s(T_w) ~~(w\in J).$$ Define $\nu(A)=\mathcal{M}^s(E\cap A)$, for all $A\in J_{\infty}$. Then $\nu([w])\leq c_1\phi^s(T_w)$  for any $w\in J$, and $0<\nu(J_{\infty})=\mathcal{M}^s(E)<\infty$.
\begin{eqnarray*}
	\int_{J_{\infty}}\int_{J_{\infty}}\int_{\textbf{b}\in B_{r}^{nM}}\frac{d\mu^{nM}(\textbf{b})d\nu(\omega)d\nu(\omega')}{\|x_\omega(\textbf{b})-x_{\omega'}(\textbf{b})\|^s}
	&\leq&c\int_{J_{\infty}}\int_{J_{\infty}}\frac{d\nu(w)d\nu(w')}{\phi^t(T_{w\wedge w'})} (\text{ by Lemma }\ref{lemma4})\\
	(\text{let }\upsilon=w\wedge w')&=&c\sum_{\upsilon\in J}\sum_{1\leq i\neq j\leq M}\frac{\nu([\upsilon,i])\nu([\upsilon,j])}{\phi^t(T_{\upsilon})}\\
	(\text{by }\nu([\upsilon])=\sum_{1\leq i\leq M}\nu([\upsilon,i]))&\leq&c\sum_{\upsilon\in J}\frac{\nu([\upsilon])^2}{\phi^t(T_{\upsilon})}\\
	&\leq&c\cdot c_1\sum_{1\leq r}\sum_{\upsilon\in J_r}\frac{\phi^s(T_{\upsilon})\nu([\upsilon])}{\phi^t(T_{\upsilon})}\\
	&\leq&c\cdot c_1\sum_{1\leq r}\sum_{\upsilon\in J_r}a^{r(s-t)}\nu([\upsilon])\\
	&=&c\cdot c_1 \frac{a^{s-t}\nu(J_{\infty})}{1-a^{s-t}}<\infty.
\end{eqnarray*}
Where $a$ is defined in Proposion 4.1 with $a<1$. Using Lemma \ref{lemma6}, we have for $\mu^{nM}$-a.e. $\textbf{b}\in B^{nM}_R$, we have $\dim_HK(\textbf{b})\geq t$. Let $R\rightarrow\infty$, we have for $\mu^{nM}$-a.e. $\textbf{b}\in B^{nM}_r$, $\dim_HK(\textbf{b})\geq \min\{n,d(T_1,\cdots,T_M)$. \\
Then, if $\dim_HK(\textbf{b})=\min\{n,d(T_1,\cdots,T_M)\}$, since $\underline {\dim}_BK(\textbf{b})\geq\dim_HK(\textbf{b})$ and $\overline{\dim}_BK(\textbf{b})\leq d(T_1, T_2,\cdots,T_M)$ for all $\textbf{b}\in \mathbb{F}^{nM}$ which we proved before. It follows that $\dim_B K(\textbf{b})=\min\{n,d(T_1, T_2,\cdots,T_M)\}$ for $\mu^{nM}$-a.e. $\textbf{b}\in \mathbb{F}^{nM}.$\hfill$\Box$

\end{document}